\documentclass[11pt]{article}
\usepackage{amsfonts,epsf,amsmath,amssymb,tikz,url,cprotect}

\newtheorem{theorem}{\bf Theorem}
\newtheorem{prop}{\bf Proposition}
\newtheorem{que}{Question}

\newcommand{\qed}{\hfill $\square$ \bigskip}

\author{
Jernej Azarija \\
Institute of Mathematics, Physics and Mechanics \\
Jadranska 19, 1000 Ljubljana, Slovenia \\
jernej.azarija@imfm.si
}

\date{\today}

\title{A short note on a short remark of Graham and Lov\'{a}sz}

\begin{document}
\maketitle

\begin{abstract}
Let $D$ be the distance matrix of a connected graph $G$ and let $n_{-}(G), n_{+}(G)$ be the number of strictly negative and positive eigenvalues of $D$ respectively. It was remarked in \cite{Graham} that it is not known whether there is a graph for which $n_{+}(G) > n_{-}(G).$ In this note we show that there exists an infinite number of graphs satisfying the stated inequality, namely the {\em conference graphs} of order $> 9.$ A large representative of this class being the {\em Paley graphs}. The result is obtained by derving the eigenvalues of the distance matrix of a strongly-regular graph.
\end{abstract}


\noindent {\bf AMS Subj. Class. (2010)}: 05C12, 05C50 

\section{Introduction}

In 1971 Graham and Pollak studied a problem related to a switching task performed at Bell Systems \cite{PolGH}. The underlying problem was modeled with a graph $G = (V,E)$ and the task was to find the 'shortest' labeling   $l:V \mapsto \{0,1,*\}^{n}$ of its vertices such that the Hamming distance of the labels matches the distance of the vertices. That is $$ h(l(u), l(v)) = d(u,v) \quad \mbox{for all } u,v \in V,$$ where $h(x,y)$ denotes the {\em Hamming} distance between the binary strings $x$ and $y.$ Using the theory of quadratic forms they have found a lower bound for $n$ in terms of the number of negative $n_{-}(G)$ and positive $n_{+}(G)$ eigenvalues of the {\em distance matrix} $D(G)$ of a graph $G.$ In addition they have shown the surprising identity \begin{equation} \det(D(T)) = (-1)^{n-1} (n-1)2^{n-2} \label{EQ} \end{equation} for any tree $T$ of order $n.$ For a more in-depth treatment of this problem see \cite[p.77]{CombB} and \cite[p.168]{Hammack-2011} for a more general exposition of such problems. 

The result for the determinant of the distance matrix of a tree was extended further by Graham and Lov\'{a}sz \cite{Graham} by generalizing (\ref{EQ}) to all coefficients of the characteristic polynomial of $D(T).$ In the same paper they remarked that it is not known whether a graph $G$ exists such that $$n_{+}(G) > n_{-}(G).$$ In this paper we show that such graphs exist indeed. In virtue of the fact that their distance matrix has more positive than negative eigenvalues we call them {\em optimistic graphs.}

\section{Eigenvalues of the distance matrix of strongly regular graphs}

It can be verified by means of a computer that there is no optimistic graph on up to 11 vertices. However, we show that there is such a graph on 13 vertices.

We recall that a {\em strongly-regular} graph with parameters $(v,k , \lambda, \mu)$ is a $k$-regular graph of order $v$ in which any pair of adjacent vertices has $\lambda$ vertices in common while any pair of non-adjacent vertices share $\mu$ common neighbors. For our purposes we are interested in {\em conference graph} which are strongly regular graphs with parameters $$k = \frac{v-1}{2}, \lambda = \frac{v-5}{4} \mbox{ and } \mu = \frac{v-1}{4}.$$ It follows from the theory of strongly regular graphs \cite[p.217]{Godsil} that the adjacency matrix of a strongly regular graph with such parameters has eigenvalues \begin{equation}k,\frac{-1-\sqrt{v}}{2}, \frac{-1+\sqrt{v}}{2} \label{EQeig} \end{equation} of multiplicities $1,\frac{v-1}{2},\frac{v-1}{2}$ respectively. In the following theorem we show that the distance matrix of every conference graph of order $v > 5$ has precisely $\frac{v+1}{2}$ positive eigenvalues, by computing the eigenvalues of the distance matrix of a strongly regular graph.

\begin{theorem} \label{T1}
Every conference graph of order $v > 9$ is optimistic.
\end{theorem}

\begin{proof}
Let $G$ be a strongly-regular graph with parameters $(v,k,\lambda,\mu)$ and distance matrix $D.$ Let $u,v$ be two distinct vertices of $G.$ By definition, the number of 2-walks between $u$ and $v$ is $\lambda$ if $u$ and $v$ are adjacent and $\mu$ otherwise. If $A$ denotes the {\em adjacency} matrix of $G,$ then since $G$ has diameter 2 its distance matrix is \begin{equation}D = A + \frac{2}{\mu}\cdot(A^2 - k I - \lambda A) = \frac{2}{\mu} A^2 + (1 - \frac{2\lambda}{\mu})A - \frac{2k}{\mu}I  \label{EQ1}.\end{equation} For a conference graph this simplifies further to $$D = \frac{1}{v-1}(8 A^2 + (9-v) A) - 4I.$$ Since this is a polynomial in $A$ we obtain the eigenvalues of $D$ by plugging the eigenvalues (\ref{EQeig}) in the above equation. We thus infer that if $G$ is a conference graph it has eigenvalues
$$\frac{3}{2}(v-1),\frac{-3-\sqrt{v}}{2}, \frac{-3+\sqrt{v}}{2},$$ and since for $v > 9$ precisely $\frac{v+1}{2}$ of the eigenvalues are positive we deduce our claim. \qed
\end{proof}

Before commenting the result we state the relation derived in the above proof that can be used to compute the eigenvalues of the distance matrix of a strongly regular graph.

\begin{prop}\label{PropID}
Let $G$ be a connected strongly regular graph with parameters $(v,k,\lambda, \mu),$ adjacency matrix $A$ and distance matrix $D.$ Then $\nu$ is an eigenvalue of $A$ if and only if $\frac{2}{\mu} \nu^2 + (1-\frac{2\lambda}{\mu})\nu - \frac{2k}{\mu}$ is an eigenvalue of $D.$
\end{prop}

For a prime power $q^k$ where $q \equiv 1 \pmod{4}$ the Paley graph $P(q^k)$ is defined as the graph with vertex set $\mathbb{F}_{q^k}$ two vertices $a,b$ being adjacent if and only if $a-b$ is a square in $\mathbb{F}_{q^k}.$ Theorem \ref{T1} implies that $$n_{+}(P(q^k)) > n_{-}(P(q^k)).$$ In particular we have shown that for any conference graph of order $v > 5$ $$n_{+}(G) = n_{-}(G)+1.$$  We can find additional optimistic graphs by extending our search from here in virtually every direction. For example there are many other optimistic strongly-regular graphs, one of them being the Hall-Janko graph $\mathcal{H}$ with parameters $(100,36,14,12).$ It can be verified using (\ref{EQ1}) that $n_{+}(\mathcal{H}) = n_{-}(\mathcal{H})+28.$ There are additional self-complementary optimistic graphs as well. A non-regular example is shown on Figure \ref{SC}. It can be checked that (excluding the Paley graphs) there are precisely 6 more optimistic self-complementary graphs on up to $17$ vertices, all of them satisfying the relation $n_{+}(G) = n_{-}(G)+1$ and having diameter 2. However there are examples of optimistic graphs of higher diameter as well. The smallest vertex-transitive optimistic graphs having diameter 3 and 4 respectively are depicted on Figure \ref{HD}. Their graph6 representation being \verb+UsaCC@u]QwLODoIo@wBI?So?{??@~??lw?h{?Bv?+ and \verb+YsP@?__C?A?O@@AA?GOCA?C??_G?g?@O?G??@?????o_?Cc???S???g_+. For a description of the graph6 format see \cite{McKay}.

\begin{figure}
\begin{center}
\includegraphics[scale=0.7]{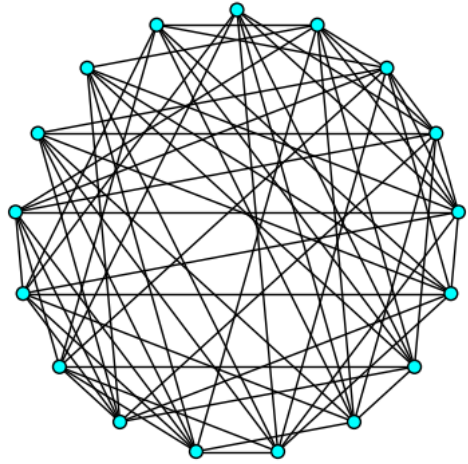} 
\end{center}
\cprotect\caption{An optimistic self-complementary graph with {\em graph6} string representation \verb+P?BMP{}kmh[X\\SjCrHisfYJ[+}
\label{SC}
\end{figure}

\begin{figure}
\begin{center}
\includegraphics[scale=0.7]{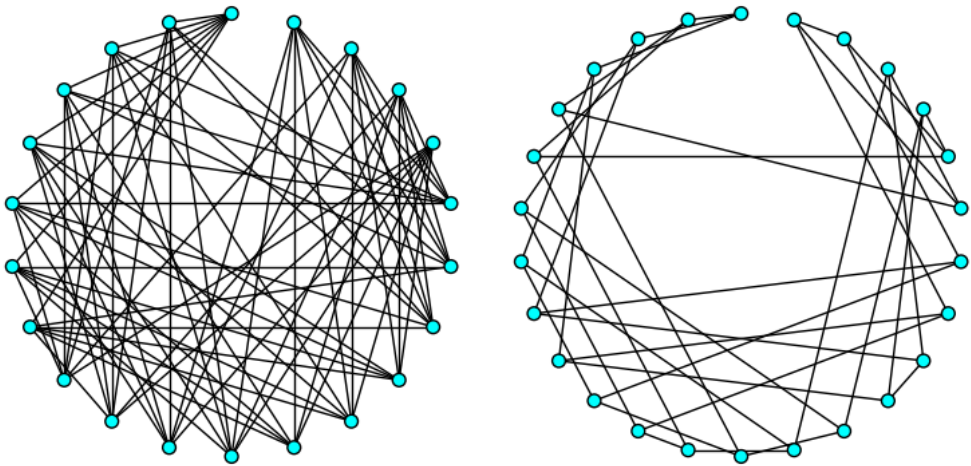} 
\end{center}
\caption{Smallest vertex-transitive optimistic graphs of diameter 3 and 4 respectively.}
\label{HD}
\end{figure}

In the process of reviewing the paper one of the reviewers asked if the gap between $n_{+}(G)$ and $n_{-}(G)$ can be made arbitrarily large. In particular the referee asked whether one can construct graphs of order $n$ such that $$n_{+}(G) - n_{-}(G) \geq c \log{n}$$ for some constant $c > 0.$ As it turns out one can obtain an even larger gap of order $n.$ It is well known \cite{Browers} that there exist a strongly regular graph with parameters $(m^2, 3(m-1), m, 6)$ for any $m > 2.$ From Proposition \ref{PropID} we can deduce that the eigenvalues of its distance matrix are $1, 1-m, m(2m-3)+1$ with respective multiplicities $m^2-3m+2, 3m-3, 1.$ Hence for every such graph $G$ we have $$n_{+}(G) - n_{-}(G) = m^2-6m+6.$$ 

As remarked at the beginning of the section a computer search indicated that there is no optimistic graph on at most 11 vertices. Since there are too many graphs of order 12 to be inspected trivially we leave the following for further research

\begin{que}
Is there an optimistic graph of order 12? If not, is $P(13)$ the unique optimistic graph on 13 vertices?
\end{que}

\section{Acknowledgements}

The author is thankful to Sandi Klav\v{z}ar for constructive discussions and to Nejc Trdin for kindly sharing his computational resources. Finally, the author would like to thank the anonymous reviewers for their useful suggestions.

\end{document}